\documentclass{article}
\usepackage{amsmath}

\newtheorem{theorem}{Theorem}

\newtheorem{corollary}[theorem]{Corollary}

\newtheorem{definition}[theorem]{Definition}

\newtheorem{lemma}[theorem]{Lemma}

\newenvironment{proof}[1][Proof]{\noindent\textbf{#1.} }{\ \rule{0.5em}{0.5em}}
\input{tcilatex}
\begin{document}

\title{Matrix Representation of Bi-Periodic Jacobsthal Sequence }
\author{Sukran Uygun\thanks{%
e-mail: suygun@gantep.edu.tr}, Evans Owusu\thanks{%
eo43241@mail2.gantep.edu.tr/owusuevans14@gmail.com} \\
Department of Mathematics - Faculty of Science and Arts,\\
Gaziantep University, 27310, Gaziantep, Turkey}
\maketitle

\begin{abstract}
In this paper, we bring into light the matrix representation of bi-periodic
Jacobsthal sequence, which we shall call the bi-periodic Jacobsthal Matrix
sequence. We define it as 
\begin{equation*}
\ \ \text{ \ }J_{n}=\left\{ 
\begin{array}{c}
bJ_{n-1}+2J_{n-2},\text{\ \ \ \ if }n\text{ is even} \\ 
aJ_{n-1}+2J_{n-2},\text{\ \ \ \ if }n\text{ is odd\ }%
\end{array}%
\right. \text{\ \ }n\geq 2,
\end{equation*}%
with initial conditions $\ J_{0}=I$ \ identity matrix $,\ J_{1}=\left( 
\begin{array}{cc}
b & 2\frac{b}{a} \\ 
1 & 0%
\end{array}%
\right) $.

We obtained the nth general term of this new matrix sequence. By studying
the properties of this new matrix sequence, the well-known Cassini or
Simpson's formula was obtained. We then proceeded to find its generating
function as well as the Binet formula. Some new properties and two summation
formulas for this new generalized matrix sequence are also given.


\end{abstract}

\section{Introduction}

The increasing applications of integer sequences such as Fibonacci, Lucas,
Jacobsthal, Jacobsthal Lucas, Pell etc in the various fields of science and
arts can not be overemphasized. For example, the ratio of two consecutive
Fibonacci numbers converges to what is widely known as the Golden ratio
whose applications appear in many research areas, particularly in Physics,
Engineering, Architecture, Nature and Art. Horadam[1]

The same can easily be said for Jacobsthal sequence. For instance, it is
known that Microcontrollers and other computers change the flow of execution
of a program using conditional instructions. Along with branch instructions,
some microcontrollers use skip instructions which conditionally bypass the
next instruction which boil down to being useful for one case out of the
four possibilities on 2 bits, 3 cases on 3 bits, 5 cases on 4 bits, 11 on 5
bits, 21 on 6 bits, 43 on 7 bits, 85 on 8 bits \ and continue in that order,
which are exactly the Jacobsthal numbers [14].

Now, The classical Jacobsthal sequence $\left\{ j_{n}\right\} _{n=0}^{\infty
}$ which was named after the German mathematician Ernst Jacobsthal is
defined recursively by the relation $j_{n}=$ $j_{n-1}+2j_{n-2}$ with initial
conditions $j_{0}=0,$ $j_{1}=1.$ The other related sequence is the
Jacobsthal Lucas sequence $\left\{ c_{n}\right\} _{n=0}^{\infty }$ which
satisfies the same recurrence relation, that is $c_{n}=c_{n-1}+2c_{n-2}$ but
with different initial conditions $c_{0}=2,$ $c_{1}=1.$ Applications of
these two sequences to curves can be found in [13]

There are many generalization in literature on the above well-known integer
sequences many of which can be found in our references. For example, in
[6,7], Edson and Yayenie defined the bi-periodic Fibonacci sequence as%
\begin{equation*}
q_{n}=\left\{ 
\begin{array}{c}
aq_{n-1}+q_{n-2},\text{ \ if }n\text{ is even} \\ 
bq_{n-1}+q_{n-2},\text{\ if }n\text{ is odd\ \ }%
\end{array}%
\right. \ \ n\geq 2
\end{equation*}%
with initial conditions $q_{0}=0,q_{1}=1$. After this, Bilgici [8] defined
the bi-periodic Lucas sequence as%
\begin{equation*}
l_{n}=\left\{ 
\begin{array}{c}
al_{n-1}+l_{n-2},\text{ \ if }n\text{ is odd } \\ 
bl_{n-1}+l_{n-2},\text{\ \ if }n\text{ is even}%
\end{array}%
\right. \ \ n\geq 2
\end{equation*}%
with initial conditions $l_{0}=2,l_{1}=a$. Bilgici also extablished some
relationships between the bi-periodic Fibonacci and Lucas numbers.

In [12 ], we defined the bi-periodic Jacobsthal sequence as%
\begin{equation*}
\hat{\jmath}_{n}=\left\{ 
\begin{array}{c}
a\hat{\jmath}_{n-1}+2\hat{\jmath}_{n-2},\text{ \ if }n\text{ is even} \\ 
b\hat{\jmath}_{n-1}+2\hat{\jmath}_{n-2},\text{\ \ if }n\text{ is odd\ }%
\end{array}%
\right. \text{\ \ \ }n\geq 2,
\end{equation*}%
\qquad\ \qquad

with initial conditions$\ \hat{\jmath}_{0}=0,\ \hat{\jmath}_{1}=1$. In [15
], we also brought into light bi-periodic Jacobsthal Lucas sequence $\left\{
C_{n}\right\} _{n=0}^{\infty }$ as%
\begin{equation*}
C_{n}=\left\{ 
\begin{array}{c}
bC_{n-1}+2C_{n-2},\text{\ if }n\text{ is even} \\ 
aC_{n-1}+2C_{n-2},\text{\ if }n\text{ is odd}%
\end{array}%
\right. \text{\ \ }n\geq 2.
\end{equation*}%
\qquad with initial conditions $C_{0}=2,\ C_{1}=a.$ The direct relationship
between the bi-periodic Jacobsthal and the bi-periodic Jacobsthal Lucas
sequences were obtained as $C_{n}=2\hat{\jmath}_{n-1}+\hat{\jmath}_{n+1}$ $\
\ \ \ $and $\ \ \left( ab+8\right) \hat{\jmath}_{n}=2C_{n-1}+C_{n+1}.$

In [16], Coskun and Taskara defined the bi-periodic Fibonacci matrix
sequence as%
\begin{equation*}
\ \text{\ }F_{n}\left( a,b\right) =\left\{ 
\begin{array}{c}
aF_{n-1}\left( a,b\right) +2F_{n-2}\left( a,b\right) ,\text{ \ if }n\text{
is even} \\ 
bF_{n-1}\left( a,b\right) +2F_{n-2}\left( a,b\right) ,\text{\ \ if }n\text{
is odd\ }%
\end{array}%
\right. \text{\ \ \ }n\geq 2,
\end{equation*}

with the initial conditions given as%
\begin{equation*}
\ F_{0}\left( a,b\right) =\left( 
\begin{array}{cc}
1 & 0 \\ 
0 & 1%
\end{array}%
\right) ,\text{ }F_{1}\left( a,b\right) =\left( 
\begin{array}{cc}
b & \frac{b}{a} \\ 
1 & 0%
\end{array}%
\right) .
\end{equation*}

They then obtained the nth general term of this matrix sequence as%
\begin{equation*}
F_{n}=\left( 
\begin{array}{cc}
\left( \frac{b}{a}\right) ^{\varepsilon (n)}q_{n+1} & \frac{b}{a}q_{n} \\ 
q_{n} & \left( \frac{b}{a}\right) ^{\varepsilon (n)}q_{n-1}%
\end{array}%
\right) ,
\end{equation*}

where $\varepsilon (m)$ is the parity function which is defined as%
\begin{equation*}
\epsilon (m)=\left\{ 
\begin{array}{c}
0,\text{ \ if }m\text{ is even} \\ 
1,\text{ \ if }m\text{ is odd}%
\end{array}%
\right.
\end{equation*}

In addiction, the authors obtained the binet formula for this sequence as%
\begin{equation*}
J_{n}=A\left( \alpha ^{n}-\beta ^{n}\right) +B\left( \alpha ^{2\left\lfloor 
\frac{n}{2}\right\rfloor +2}-\beta ^{2\left\lfloor \frac{n}{2}\right\rfloor
+2}\right)
\end{equation*}

where%
\begin{equation*}
A=\frac{\left( F_{1}\left( a,b\right) -bF_{0}\left( a,b\right) \right)
^{\epsilon (n)}\left( aF_{1}\left( a,b\right) -F_{0}\left( a,b\right)
-abF_{0}\left( a,b\right) \right) ^{1-\epsilon (n)}}{\left( ab\right)
^{\left\lfloor \frac{n}{2}\right\rfloor }\left( \alpha -\beta \right) }\text{
and\ }B=\text{\ }\frac{b^{\epsilon (n)}F_{0}\left( a,b\right) }{\left(
ab\right) ^{\left\lfloor \frac{n}{2}\right\rfloor +1}\left( \alpha -\beta
\right) }
\end{equation*}

and $\alpha =\frac{ab+\sqrt{a^{2}b^{2}+8ab}}{2}$ and $\beta =\frac{ab-\sqrt{%
a^{2}b^{2}+8ab}}{2}$ are the roots of the characteristic equation $%
x^{2}-abx-2ab=0.$ Using the Binet formula, some summations for the
bi-periodic Fibonnacci matrix sequence were also given.

In the same way, In [11] Coskun, Yilmaz and Taskara defined the bi-periodic
Lucas matrix sequence as%
\begin{equation}
\ \text{\ }L_{n}\left( a,b\right) =\left\{ 
\begin{array}{c}
aL_{n-1}\left( a,b\right) +2L_{n-2}\left( a,b\right) ,\text{ \ \ \ \ \ \ \ \
\ \ \ \ if }n\text{ is even} \\ 
bL_{n-1}\left( a,b\right) +2L_{n-2}\left( a,b\right) ,\text{\ \ \ \ \ \ \ \
\ \ \ \ if }n\text{ is odd}%
\end{array}%
\right. \text{\ \ \ }n\geq 2,  \tag{1}
\end{equation}

with the initial conditions given as%
\begin{equation*}
\ L_{0}\left( a,b\right) =\left( 
\begin{array}{cc}
a & 2 \\ 
2\frac{a}{b} & -a%
\end{array}%
\right) ,\text{ }L_{1}\left( a,b\right) =\left( 
\begin{array}{cc}
a^{2}+2\frac{a}{b} & a \\ 
\frac{a^{2}}{b} & 2\frac{a}{b}%
\end{array}%
\right) .
\end{equation*}

They then obtained the nth general term of this matrix sequence as%
\begin{equation*}
L_{n}\left( a,b\right) =\left( 
\begin{array}{cc}
\left( \frac{a}{b}\right) ^{\varepsilon (n)}l_{n+1} & l_{n} \\ 
\frac{a}{b}l_{n} & \left( \frac{a}{b}\right) ^{\varepsilon (n)}l_{n-1}%
\end{array}%
\right) ,
\end{equation*}

where $\varepsilon (m)$ is the parity function which is defined as before.

In addiction, the authors obtained the binet formula for this sequence as

\begin{equation*}
L_{n}=A\alpha ^{n}+B\beta ^{n}
\end{equation*}

where

\begin{equation*}
A=\frac{bL_{1}\left( a,b\right) -\alpha L_{0}\left( a,b\right)
-abL_{0}\left( a,b\right) }{b^{\varepsilon (n)}\left( ab\right)
^{\left\lfloor \frac{n}{2}\right\rfloor }\left( \alpha -\beta \right) }\text{
\ \ \ and \ }B=\text{\ }\frac{bL_{1}\left( a,b\right) -\beta L_{0}\left(
a,b\right) -abL_{0}\left( a,b\right) }{b^{\varepsilon (n)}\left( ab\right)
^{\left\lfloor \frac{n}{2}\right\rfloor }\left( \alpha -\beta \right) }
\end{equation*}

In this paper, as first in literature, we bring into light the matrix
representation of bi-periodic Jacobsthal sequence, which we shall call the
bi-periodic Jacobsthal Matrix sequence. We will then proceed to obtain the
nth general term of this new matrix sequence. By studying the algebraic
properties of this new matrix sequence, the well-known Cassini or Simpson's
formula would be given. The generating function together with the Binet
formula are also given. Some new \ properties as well as some summation
formulas for this new generalized matrix sequence are also given.

\begin{definition}
For any two non-zero real numbers $a$ and $b$, and any number $n$ belonging
to the set of natural numbers, the bi-periodic Jacobsthal \ matrix sequence
denoted by $J_{n}\left( a,b\right) $ is defined recursively by
\end{definition}

\begin{equation*}
\ \text{\ }J_{n}\left( a,b\right) =\left\{ 
\begin{array}{c}
aJ_{n-1}\left( a,b\right) +2J_{n-2}\left( a,b\right) ,\text{ \ \ \ \ \ \ \ \
\ \ \ \ if }n\text{ is even} \\ 
bJ_{n-1}\left( a,b\right) +2J_{n-2}\left( a,b\right) ,\text{\ \ \ \ \ \ \ \
\ \ \ \ if }n\text{ is odd}%
\end{array}%
\right. \text{\ \ \ }n\geq 2,
\end{equation*}

with the initial conditions given as

\bigskip 
\begin{equation*}
\ J_{0}\left( a,b\right) =\left( 
\begin{array}{cc}
1 & 0 \\ 
0 & 1%
\end{array}%
\right) ,\text{ }J_{1}\left( a,b\right) =\left( 
\begin{array}{cc}
b & 2\frac{b}{a} \\ 
1 & 0%
\end{array}%
\right) .
\end{equation*}

For the brevity, we shall use $J_{n}$ in place of $J_{n}(a,b).$

\begin{theorem}
For any integer $n\geq 0$, we obtain the $nth$ Jacobsthal matrix sequence as
\end{theorem}

\begin{equation*}
J_{n}=\left( 
\begin{array}{cc}
\left( \frac{b}{a}\right) ^{\varepsilon (m)}j_{n+1} & 2\frac{b}{a}j_{n} \\ 
j_{n} & 2\left( \frac{b}{a}\right) ^{\varepsilon (m)}j_{n-1}%
\end{array}%
\right) .
\end{equation*}

\begin{proof}
The proof is done by means of mathematical induction. We will start by
noting from the classical Jacobsthal sequence $j_{n}$ as defined in the
introduction that $j_{0}=0,\ j_{1}=1,j_{-1}=\frac{1}{2}\ $and $j_{2}=a.$%
Hence the induction for $n=0$ and $n=1$ are respectively as follows%
\begin{equation}
\text{ }J_{0}=\left( 
\begin{array}{cc}
j_{1} & 2\frac{b}{a}j_{0} \\ 
j_{0} & 2j_{-1}%
\end{array}%
\right) =I  \notag
\end{equation}
\end{proof}

\begin{equation}
J_{1}=\left( 
\begin{array}{cc}
\left( \frac{b}{a}\right) j_{2} & 2\frac{b}{a}j_{1} \\ 
j_{1} & 2j0%
\end{array}%
\right) =\left( 
\begin{array}{cc}
b & 2\frac{b}{a} \\ 
1 & 0%
\end{array}%
\right)  \notag
\end{equation}

We now assume that the equation is true for $n=k$, where $k$ is a positive
integer, that is;

\begin{equation*}
J_{k}=\left( 
\begin{array}{cc}
\left( \frac{b}{a}\right) ^{\varepsilon (k)}j_{k+1} & 2\frac{b}{a}j_{k} \\ 
j_{k} & 2\left( \frac{b}{a}\right) ^{\varepsilon (k)}j_{k-1}%
\end{array}%
\right)
\end{equation*}

We will end the proof by showing that the equation also holds for $n=k+1;$
that is%
\begin{eqnarray*}
\ \text{\ }J_{k+1} &=&\left\{ 
\begin{array}{c}
aJ_{k}+2J_{k-1},\text{ \ \ \ \ \ \ \ \ \ \ \ \ if }k+1\text{ is even} \\ 
bJ_{k}+2J_{k-1},\text{\ \ \ \ \ \ \ \ \ \ \ \ if }k+1\text{ is odd}%
\end{array}%
\right. \ \text{\ } \\
&=&a^{\varepsilon (k)}b^{1-\varepsilon (k)}\left[ \text{\ }J_{k}+2\text{\ }%
J_{k-1}\right] \\
&=&a^{\varepsilon (k)}b^{1-\varepsilon (k)}\left( 
\begin{array}{cc}
\left( \frac{b}{a}\right) ^{\varepsilon (k)}j_{k+1} & 2\frac{b}{a}j_{k} \\ 
j_{k} & 2\left( \frac{b}{a}\right) ^{\varepsilon (k)}j_{k-1}%
\end{array}%
\right) +2\left( 
\begin{array}{cc}
\left( \frac{b}{a}\right) ^{\varepsilon (k)}j_{k} & 2\frac{b}{a}j_{k-1} \\ 
j_{k-1} & 2\left( \frac{b}{a}\right) ^{\varepsilon (k)}j_{k-2}%
\end{array}%
\right) \\
&=&\left\{ 
\begin{array}{c}
\left( 
\begin{array}{cc}
bj_{k+1}+2\frac{b}{a}j_{k}\text{ } & 2\left( \frac{b}{a}\right) ^{2}j_{k}+4%
\frac{b}{a}j_{k-1}\text{ } \\ 
bj_{k}+2j_{k-1}\text{ } & 2bj_{k=1}+4\frac{b}{a}j_{k=2}%
\end{array}%
\text{\ }\right) =\left( 
\begin{array}{cc}
\frac{b}{a}j_{k+2} & 2\frac{b}{a}j_{k+1} \\ 
j_{k+1} & 2\frac{b}{a}j_{k}%
\end{array}%
\right) \text{\ \ }k\text{ even} \\ 
\left( 
\begin{array}{cc}
a\frac{b}{a}j_{k+1}+2j_{k}\text{ } & 2a\frac{b}{a}j_{k}+4\frac{b}{a}j_{k-1}%
\text{ } \\ 
aj_{k}+2j_{k-1}\text{ } & 2a\frac{b}{a}j_{k=1}+4j_{k=2}%
\end{array}%
\text{\ }\right) \text{\ \ \ \ }=\left( 
\begin{array}{cc}
j_{k+2} & 2\frac{b}{a}j_{k+1} \\ 
j_{k+1} & 2j_{k}%
\end{array}%
\right) \text{\ \ \ \ \ \ \ }k\text{ odd}%
\end{array}%
\right. \\
&=&\ \text{\ }\left( 
\begin{array}{cc}
\left( \frac{b}{a}\right) ^{\varepsilon (k+1)}j_{k+2} & 2\frac{b}{a}j_{k+1}
\\ 
j_{k+1} & 2\left( \frac{b}{a}\right) ^{\varepsilon (k)}j_{k}%
\end{array}%
\right) .\text{ \ }
\end{eqnarray*}

\begin{lemma}
For any integer $m\geq 0$, we obtain%
\begin{eqnarray*}
J_{2m} &=&(ab+4)J_{2m-2}-4J_{2m-4}, \\
J_{2m+1} &=&(ab+4)J_{2m-1}-4J_{2m-3}.
\end{eqnarray*}
\end{lemma}

\begin{proof}
The proof can easily be obtained by using the above definition of the
bi-periodic Jacobsthal matrix sequence.
\end{proof}

\begin{theorem}
For any positive integer $n$, we have%
\begin{equation*}
\det \left[ J_{n}\right] =2^{n}\left( -\frac{b}{a}\right) ^{\varepsilon (n)}
\end{equation*}
\end{theorem}

\begin{proof}
\begin{eqnarray*}
\det \left[ J_{0}\right] &=&\det \left( 
\begin{array}{cc}
1 & 0 \\ 
0 & 1%
\end{array}%
\right) =1. \\
\det \left[ J_{1}\right] &=&\det \left( 
\begin{array}{cc}
b & 2\frac{b}{a} \\ 
1 & 0%
\end{array}%
\right) =-2\frac{b}{a} \\
\det \left[ J_{2}\right] &=&\det \left( 
\begin{array}{cc}
ab+2 & 2b \\ 
a & 2%
\end{array}%
\right) =4 \\
\det \left[ J_{3}\right] &=&\det \left( 
\begin{array}{cc}
a^{2}b+4b & 2b^{2}+4\frac{b}{a} \\ 
ab+2 & 2b%
\end{array}%
\right) =-8\frac{b}{a} \\
\det \left[ J_{4}\right] &=&\det \left( 
\begin{array}{cc}
a^{2}b^{2}+6ab+4 & 2ab^{2}+8b \\ 
a^{2}b+4a & 2ab+4%
\end{array}%
\right) =16 \\
\det \left[ J_{5}\right] &=&\det \left( 
\begin{array}{cc}
a^{2}b^{3}+8ab^{2}+12b & 2ab^{3}+12b^{2}+8\frac{b}{a} \\ 
a^{2}b^{2}+6ab+4 & 2ab^{2}+8b%
\end{array}%
\right) =-32\frac{b}{a} \\
\det \left[ J_{6}\right] &=&\det \left( 
\begin{array}{cc}
a^{3}b^{3}+10a^{2}b^{2}+24ab+8 & 2a^{2}b^{3}+16ab^{2}+24b \\ 
a^{3}b^{2}+8a^{2}b+12a & 2a^{2}b^{2}+12ab+8%
\end{array}%
\right) =64 \\
... &=&\text{ the order continues.}
\end{eqnarray*}
\end{proof}

Therefore the above procedure can be iterated as

\begin{equation*}
\det \left[ J_{n}\right] =\left\{ 
\begin{array}{c}
2^{n}\text{ \ \ \ \ \ \ \ \ \ \ \ \ \ \ \ \ \ if\ }n\text{ even} \\ 
2^{n}\left( -\frac{b}{a}\right) \text{\ \ \ \ \ \ \ if \ }n\text{ odd}%
\end{array}%
\right.
\end{equation*}

which can be condensed using the parity function as

\begin{equation*}
\det \left[ J_{n}\right] =2^{n}\left( -\frac{b}{a}\right) ^{\varepsilon (n)}
\end{equation*}

which completes the proof

\begin{corollary}
$\left( CASSINI^{\prime }S\text{ }IDENTITY\text{ }/SIMPSON^{\prime }S\text{ }%
FORMULA\right) $
\end{corollary}

\begin{theorem}
This identity is obtain by a mere comparison of the determinants of the
bi-periodic Jacobsthal matrix sequence with the determinant of its
individual terms condensed together as shown in the immediate theorem above.
Hence our casini's identity is given by
\end{theorem}

\begin{equation*}
2\left( \frac{b}{a}\right) ^{2\varepsilon (m)}j_{n-1}j_{n+1}-2\frac{b}{a}%
j_{n}^{2}=2^{n}\left( -\frac{b}{a}\right) ^{\varepsilon (n)}
\end{equation*}

which can be simplified as

\begin{equation*}
\left( \frac{b}{a}\right) ^{\varepsilon (m)}j_{n-1}j_{n+1}-\left( \frac{b}{a}%
\right) ^{1-\varepsilon (m)}j_{n}^{2}=\left( -1\right) ^{\varepsilon
(n)}2^{n-1}
\end{equation*}

\begin{theorem}
$\left( GENERATING\text{ }FUNCTION\right) $

The generating function for the bi-periodic Jacobsthal matrix sequence is
obtained as

\begin{equation*}
\underset{m=0}{\overset{\infty }{\sum }}J_{m}x^{m}=\frac{J_{0}+J_{1}x+\left[
aJ_{1}-\left( ab+2\right) J_{0}\right] x^{2}+\left[ 2bJ_{0}-2J_{1}\right]
x^{3}}{1-(ab+4)x^{2}+4x^{4}}.
\end{equation*}
\end{theorem}

which is expressed in component form as

\begin{equation*}
\underset{m=0}{\overset{\infty }{\sum }}J_{m}x^{m}=\frac{1}{%
1-(ab+4)x^{2}+4x^{4}}\left( 
\begin{array}{cc}
1+bx-2x^{2} & 2\frac{b}{a}x+2bx^{2}-4\frac{b}{a}x^{3} \\ 
x+ax^{2}-2x^{3} & 1-\left( ab+2\right) x^{2}+2bx^{3}%
\end{array}%
\right) .
\end{equation*}

\begin{proof}
We divide the series into two parts%
\begin{equation*}
J(x)=\underset{m=0}{\overset{\infty }{\sum }}J_{m}x^{m}=\underset{m=0}{%
\overset{\infty }{\sum }}J_{2m}x^{2m}+\underset{m=0}{\overset{\infty }{\sum }%
}J_{2m+1}x^{2m+1}.
\end{equation*}%
We simplify the even part of the above series as follows%
\begin{equation*}
J_{0}(x)=\underset{m=0}{\overset{\infty }{\sum }}%
J_{2m}x^{2m}=J_{0}+J_{2}x^{2}+\underset{m=2}{\overset{\infty }{\sum }}%
J_{2m}x^{2m}
\end{equation*}%
By multiplying through by $(ab+4)x^{2}$ and $4x^{4}$ respectively, we have%
\begin{equation*}
(ab+4)x^{2}J_{0}(x)=(ab+4)J_{0}x^{2}+(ab+4)\underset{m=2}{\overset{\infty }{%
\sum }}J_{2m-2}x^{2m}
\end{equation*}%
and%
\begin{equation*}
4x^{4}J_{0}(x)=\underset{m=2}{\overset{\infty }{4\sum }}J_{2m-4}x^{2m}.
\end{equation*}%
Hence it follows that,%
\begin{eqnarray*}
\left[ 1-(ab+4)x^{2}+4x^{4}\right] J_{0}(x)
&=&J_{0}+J_{2}x^{2}-(ab+4)J_{0}x^{2} \\
&&+\underset{m=2}{\overset{\infty }{\sum }}\left[
J_{2m}-(ab+4)J_{2m-2}+4J_{2m-4}\right] x^{2m}
\end{eqnarray*}%
By using Lemma 1, we obtained that;%
\begin{equation*}
J_{0}(x)=\frac{J_{0}+J_{2}x^{2}-(ab+4)J_{0}x^{2}}{1-(ab+4)x^{2}+4x^{4}}.
\end{equation*}%
Similarly, the odd part of the above series is simplified as follows%
\begin{equation*}
J_{1}(x)=\underset{m=0}{\overset{\infty }{\sum }}%
J_{2m+1}x^{2m+1}=J_{1}x+J_{3}x^{3}+\underset{m=2}{\overset{\infty }{\sum }}%
J_{2m+1}x^{2m+1}
\end{equation*}%
By multiplying through by $(ab+4)x^{2}$ and $4x^{4}$ respectively, we obtain%
\begin{equation*}
(ab+4)x^{2}J_{1}(x)=(ab+4)J_{1}\left( a,b\right) x^{3}+(ab+4)\underset{m=2}{%
\overset{\infty }{\sum }}J_{2m-1}\left( a,b\right) x^{2m+1}.
\end{equation*}%
and%
\begin{equation*}
4x^{4}J_{1}(x)=\underset{m=2}{\overset{\infty }{4\sum }}J_{2m-3}x^{2m+1}.
\end{equation*}%
Hence it follows that,%
\begin{eqnarray*}
\left[ 1-(ab+4)x^{2}+4x^{4}\right] J_{1}(x)
&=&J_{1}x+J_{3}x^{3}-(ab+4)J_{1}x^{3} \\
&&\underset{m=2}{\overset{\infty }{+\sum }}\left[
J_{2m+1}-(ab+4)J_{2m-1}+4J_{2m-3}\right] x^{2m+1}
\end{eqnarray*}%
By using Lemma 1, we obtained that;%
\begin{equation*}
J_{1}(x)=\frac{J_{1}x+J_{3}x^{3}-(ab+4)J_{1}x^{3}}{1-(ab+4)x^{2}+4x^{4}}.
\end{equation*}%
By combining the two results $\ \left[ J(x)=J_{0}(x)+J_{1}(x)\right] ,$ we
have%
\begin{equation*}
J(x)=\frac{%
J_{0}+J_{2}x^{2}-(ab+4)J_{0}x^{2}+J_{1}x+J_{3}x^{3}-(ab+4)J_{1}x^{3}}{%
1-(ab+4)x^{2}+4x^{4}}.
\end{equation*}%
which can be simplified using definition $\left[ 1\right] $ as 
\begin{equation*}
J(x)=\frac{J_{0}+J_{1}x+\left[ aJ_{1}-\left( ab+2\right) J_{0}\right] x^{2}+%
\left[ 2bJ_{0}-2J_{1}\right] x^{3}}{1-(ab+4)x^{2}+4x^{4}}.
\end{equation*}%
By substituting the matrix components of $J_{0}$ and $J_{1}$ as given in
definition $\left[ 1\right] $ and simplifying, we obtain%
\begin{equation*}
J(x)=\frac{1}{1-(ab+4)x^{2}+4x^{4}}\left( 
\begin{array}{cc}
1+bx-2x^{2} & 2\frac{b}{a}x+2bx^{2}-4\frac{b}{a}x^{3} \\ 
x+ax^{2}-2x^{3} & 1-\left( ab+2\right) x^{2}+2bx^{3}%
\end{array}%
\right) .
\end{equation*}%
which completes the proof.
\end{proof}

We would like to show another proof of this theorem.

\textbf{(2}) The generating function for $J(x)$ is represented in power
series by%
\begin{equation*}
J(x)=\underset{m=0}{\overset{\infty }{\sum }}%
J_{m}x^{m}=J_{0}+J_{1}x+....+J_{k}x^{k}+...
\end{equation*}%
By multiplying through this series by $bx$ and $2x^{2}$ respectively, we
obtain%
\begin{equation*}
bxJ(x)=\underset{m=0}{\overset{\infty }{b\sum }}J_{m}x^{m+1}=\underset{m=1}{%
\overset{\infty }{b\sum }}J_{m-1}x^{m}
\end{equation*}%
and,%
\begin{equation*}
2x^{2}J(x)=\underset{m=0}{\overset{\infty }{2\sum }}J_{m}x^{m+2}=\underset{%
m=2}{\overset{\infty }{2\sum }}J_{m-2}x^{m}.
\end{equation*}%
therefore we have%
\begin{eqnarray*}
(1-bx-2x^{2})J(x) &=&J_{0}+xJ_{1}+bxJ_{0} \\
&&+\underset{m=2}{\overset{\infty }{\sum }}(J_{m}-bJ_{m-1}-2J_{m-2})x^{m} \\
&=&J_{0}+xJ_{1}+bxJ_{0} \\
&&+\underset{m=1}{\overset{\infty }{\sum }}(J_{2m}-bJ_{2m-1}-2J_{2m-2})x^{2m}
\end{eqnarray*}%
\qquad \qquad\ From Lemma 1, $J_{2m}=aJ_{2m-1}+2J_{2m-2},$ therefore%
\begin{equation*}
(1-bx-2x^{2})J(x)=J_{0}+xJ_{1}+bxJ_{0}+\underset{m=1}{\overset{\infty }{\sum 
}}(a-b)J_{2m-1}x^{2m}.
\end{equation*}%
\begin{equation*}
(1-bx-2x^{2})J(x)=J_{0}+xJ_{1}+bxJ_{0}+(a-b)x\underset{m=1}{\overset{\infty }%
{\sum }}J_{2m-1}x^{2m-1}.
\end{equation*}%
Now lets define $j(x)$ as 
\begin{equation*}
j(x)=\underset{m=1}{\overset{\infty }{\sum }}J_{2m-1}x^{2m-1}.
\end{equation*}%
Simplying $j(x)$ in the same way as above and using lemma 1 gives;%
\begin{eqnarray*}
(1-(ab+4)x^{2}+4x^{4})j(x) &=&\underset{m=1}{\overset{\infty }{\sum }}%
J_{2m-1}x^{2m-1}-(ab+4)\underset{m=2}{\overset{\infty }{\sum }}%
J_{2m-3}x^{2m-1} \\
&&+4\underset{m=3}{\overset{\infty }{\sum }}J_{2m-5}x^{2m-1} \\
&=&(J_{1}x+J_{3}x^{3})-(ab+4)J_{1}x^{3} \\
&&+\underset{m=3}{\overset{\infty }{\sum }}%
(J_{2m-1}-(ab+4)J_{2m-3}+4J_{2m-5})x^{2m-1} \\
&=&J_{1}x+J_{3}x^{3}-(ab+4)J_{1}x^{3}+0.
\end{eqnarray*}%
\begin{equation*}
j(x)=\frac{J_{1}x+J_{3}x^{3}-(ab+4)J_{1}x^{3}}{1-(ab+4)x^{2}+4x^{4}}
\end{equation*}

Plugging $j(x)$ into $J(x)$ above gives 
\begin{equation*}
(1-bx-2x^{2})J(x)=J_{0}+xJ_{1}+bxJ_{0}+(a-b)x\left( \frac{%
J_{1}x+J_{3}x^{3}-(ab+4)J_{1}x^{3}}{1-(ab+4)x^{2}+4x^{4}}\right)
\end{equation*}

Simplifying this using the basic rules and properties of algebra, we obtain 
\begin{equation*}
J(x)=\frac{%
J_{0}+J_{2}x^{2}-(ab+4)J_{0}x^{2}+J_{1}x+J_{3}x^{3}-(ab+4)J_{1}x^{3}}{%
1-(ab+4)x^{2}+4x^{4}}.
\end{equation*}

which simplifies as 
\begin{equation*}
J(x)=\frac{J_{0}+J_{1}x+\left[ aJ_{1}-\left( ab+2\right) J_{0}\right] x^{2}+%
\left[ 2bJ_{0}-2J_{1}\right] x^{3}}{1-(ab+4)x^{2}+4x^{4}}.
\end{equation*}

Similarly the component form is obtained as.%
\begin{equation*}
J(x)=\frac{1}{1-(ab+4)x^{2}+4x^{4}}\left( 
\begin{array}{cc}
1+bx-2x^{2} & 2\frac{b}{a}x+2bx^{2}-4\frac{b}{a}x^{3} \\ 
x+ax^{2}-2x^{3} & 1-\left( ab+2\right) x^{2}+2bx^{3}%
\end{array}%
\right) .
\end{equation*}

which completes the proof.

\begin{theorem}
$\left( BINET\text{ }FORMULA\right) $

For every $n$ belonging to the set of natural numbers, the Binet formula for
the bi-periodic Jacobsthal matrix sequence is given by. 
\begin{equation*}
J_{n}=A\left( \alpha ^{n}-\beta ^{n}\right) +B\left( \alpha ^{2\left\lfloor 
\frac{n}{2}\right\rfloor +2}-\beta ^{2\left\lfloor \frac{n}{2}\right\rfloor
+2}\right) .
\end{equation*}
\end{theorem}

Where

\begin{eqnarray*}
A &=&\frac{\left( J_{1}-bJ_{0}\right) ^{\epsilon (n)}\left(
aJ_{1}-2J_{0}-abJ_{0}\right) ^{1-\epsilon (n)}}{\left( ab\right)
^{\left\lfloor \frac{n}{2}\right\rfloor }\left( \alpha -\beta \right) }\text{
\ and } \\
\text{\ }B &=&\frac{b^{\epsilon (n)}J_{0}}{\left( ab\right) ^{\left\lfloor 
\frac{n}{2}\right\rfloor +1}\left( \alpha -\beta \right) }
\end{eqnarray*}

The matrices A and B are expressed in component form as

\begin{eqnarray*}
A &=&\frac{1}{\left( ab\right) ^{\left\lfloor \frac{n}{2}\right\rfloor
}\left( \alpha -\beta \right) }\left\{ \left( 
\begin{array}{cc}
0 & 2\frac{b}{a} \\ 
1 & -b%
\end{array}%
\right) ^{\epsilon (n)}\left( 
\begin{array}{cc}
-2 & 2b \\ 
a & -2-ab%
\end{array}%
\right) ^{1-\epsilon (n)}\right\} \text{ } \\
\text{and }B &=&\frac{b^{\epsilon (n)}}{\left( ab\right) ^{\left\lfloor 
\frac{n}{2}\right\rfloor +1}\left( \alpha -\beta \right) }\left( 
\begin{array}{cc}
1 & 0 \\ 
0 & 1%
\end{array}%
\right)
\end{eqnarray*}

\begin{proof}
\begin{equation*}
\Phi \left( x\right) =J_{0}+J_{1}x+\left[ aJ_{1}-\left( ab+2\right) J_{0}%
\right] x^{2}+\left[ 2bJ_{0}-2J_{1}\right] x^{3}
\end{equation*}
\end{proof}

Using partial fraction decomposition, we split $J(x)$ as%
\begin{eqnarray*}
J(x) &=&\frac{\Phi \left( x\right) }{1-(ab+4)x^{2}+4x^{4}} \\
&=&\frac{1}{4}\left[ \frac{Ax+B}{\left( x^{2}-\frac{\alpha +2}{4}\right) }+%
\frac{Cx+D}{\left( x^{2}-\frac{\beta +2}{4}\right) }\right]
\end{eqnarray*}
By solving for the constants $A,B,C$ and $D$ above, we express $J(x)$ in
partial fraction as%
\begin{equation*}
\frac{1}{4\left( \alpha -\beta \right) }\left[ \frac{\left\{ 
\begin{array}{c}
x\left\{ 2\alpha \left[ bJ_{0}-J_{1}\right] +4bJ_{0}\right\} \\ 
+\alpha \left\{ aJ_{1}-2J_{0}-abJ_{0}\right\} \\ 
+2aJ_{1}-2abJ_{0}%
\end{array}%
\right\} }{\left( x^{2}-\frac{\alpha +2}{4}\right) }+\frac{\left\{ 
\begin{array}{c}
x\left\{ 2\beta \left[ J_{1}-bJ_{0}\right] -4bJ_{0}\right\} \\ 
+\beta \left\{ abJ_{0}\left( a,b\right) +2J_{0}\left( a,b\right)
-aJ_{1}\right\} \\ 
+2abJ_{0}-2aJ_{1}%
\end{array}%
\right\} }{\left( x^{2}-\frac{\beta +2}{4}\right) }\right]
\end{equation*}%
The Maclaurin series expansion of the function $\frac{Ax+B}{x^{2}-C}$ is
expressed in the form 
\begin{equation*}
\frac{Ax+B}{x^{2}-C}=-\underset{n=0}{\overset{\infty }{\sum }}%
AC^{-n-1}x^{2n+1}-\underset{n=0}{\overset{\infty }{\sum }}BC^{-n-1}x^{2n}
\end{equation*}%
Hence $J(x)$ can be expanded and simplified as%
\begin{equation*}
J(x)=\frac{1}{4\left( \alpha -\beta \right) }\left[ 
\begin{array}{c}
-\underset{n=0}{\overset{\infty }{\sum }}\left\{ 2\alpha
bJ_{0}-J_{1}+4bJ_{0}\right\} \left( \frac{\alpha +2}{4}\right)
^{-n-1}x^{2n+1} \\ 
-\underset{n=0}{\overset{\infty }{\sum }}\left\{ 
\begin{array}{c}
\alpha \left[ aJ_{1}-2J_{0}-abJ_{0}\right] \\ 
+2aJ_{1}-2abJ_{0}%
\end{array}%
\right\} \left( \frac{\alpha +2}{4}\right) ^{-n-1}x^{2n}+ \\ 
-\underset{n=0}{\overset{\infty }{\sum }}\left\{ 
\begin{array}{c}
2\beta \left[ J_{1}-bJ_{0}\right] \\ 
-4bJ_{0}%
\end{array}%
\right\} \left( \frac{\beta +2}{4}\right) ^{-n-1}x^{2n+1} \\ 
-\underset{n=0}{\overset{\infty }{\sum }}\left\{ 
\begin{array}{c}
\beta \left\{ 
\begin{array}{c}
abJ_{0}+2J_{0} \\ 
-aJ_{1}%
\end{array}%
\right\} \\ 
+2abJ_{0}-2aJ_{1}%
\end{array}%
\right\} \left( \frac{\beta +2}{4}\right) ^{-n-1}x^{2n}+%
\end{array}%
\right]
\end{equation*}%
We obtain the even part of $J(x)$ as%
\begin{equation*}
\frac{-1}{4\left( \alpha -\beta \right) }\underset{n=0}{\overset{\infty }{%
\sum }}\left[ 
\begin{array}{c}
\left\{ 
\begin{array}{c}
\alpha \left[ aJ_{1}-2J_{0}-abJ_{0}\right] \\ 
+2aJ_{1}-2abJ_{0}%
\end{array}%
\right\} \frac{4^{n+1}}{\left( \alpha +2\right) ^{n+1}} \\ 
+\left\{ 
\begin{array}{c}
\beta \left[ abJ_{0}+2J_{0}-aJ_{1}\right] \\ 
+2abJ_{0}-2aJ_{1}%
\end{array}%
\right\} \frac{4^{n+1}}{\left( \beta +2\right) ^{n+1}}%
\end{array}%
\right] x^{2n}
\end{equation*}

Which can be simplifed as%
\begin{equation*}
\frac{-4^{n}}{\left( \alpha -\beta \right) }\underset{n=0}{\overset{\infty }{%
\sum }}\left[ \frac{\left( \beta +2\right) ^{n+1}\left\{ 
\begin{array}{c}
\alpha \left[ aJ_{1}-2J_{0}-abJ_{0}\right] \\ 
+2aJ_{1}-2abJ_{0}%
\end{array}%
\right\} +\left( \alpha +2\right) ^{n+1}\left\{ 
\begin{array}{c}
\beta \left\{ abJ_{0}+2J_{0}-aJ_{1}\right\} \\ 
+2abJ_{0}-2aJ_{1}%
\end{array}%
\right\} }{\left( \alpha +2\right) ^{n+1}\left( \beta +2\right) ^{n+1}}%
\right] x^{2n}
\end{equation*}%
From the identity that $\left( \alpha +2\right) \left( \beta +2\right) =4,$%
we have%
\begin{equation*}
\frac{1}{4\left( \alpha -\beta \right) }\underset{n=0}{\overset{\infty }{%
\sum }}\left[ \left\{ 
\begin{array}{c}
2\beta \left( \beta +2\right) ^{n}\left[ aJ_{1}-2J_{0}-abJ_{0}\right] \\ 
-\left( \beta +2\right) ^{n+1}\left[ 2aJ_{1}-2abJ_{0}\right]%
\end{array}%
\right\} +\left\{ 
\begin{array}{c}
2\alpha \left( \alpha +2\right) ^{n}\left\{ abJ_{0}+2J_{0}-aJ_{1}\right\} \\ 
-\left( \alpha +2\right) ^{n+1}\left[ 2abJ_{0}-2aJ_{1}\right]%
\end{array}%
\right\} \right] x^{2n}
\end{equation*}%
by using the identity $\left( \alpha +2\right) =\frac{\alpha ^{2}}{ab},$ we
get%
\begin{equation*}
\frac{2}{4\left( \alpha -\beta \right) }\underset{n=0}{\overset{\infty }{%
\sum }}\left( \frac{1}{ab}\right) ^{n+1}\left\{ 
\begin{array}{c}
\left[ aJ_{1}-2J_{0}-abJ_{0}\right] \left[ ab\left( \beta ^{2n+1}-\alpha
^{2n+1}\right) \right] \\ 
+\left[ aJ_{1}-abJ_{0}\right] \left[ \alpha ^{2n+2}-\beta ^{2n+2}\right]%
\end{array}%
\right\} x^{2n}
\end{equation*}%
Also, making use of the identity $ab=\alpha +\beta $ gives%
\begin{equation*}
\frac{1}{\left( \alpha -\beta \right) }\underset{n=0}{\overset{\infty }{\sum 
}}\left\{ 
\begin{array}{c}
\left( aJ_{1}-2J_{0}-abJ_{0}\right) \left\{ \frac{\alpha ^{2n}-\beta ^{2n}}{%
\left( ab\right) ^{n}\left( \alpha -\beta \right) }\right\} \\ 
+J_{0}\left\{ \frac{\alpha ^{2n+2}-\beta ^{2n+2}}{\left( ab\right)
^{n+1}\left( \alpha -\beta \right) }\right\}%
\end{array}%
\right\} x^{2n}
\end{equation*}%
In the same way, the odd part of $J(x)$ is obtained as%
\begin{equation*}
\frac{-4^{n+1}}{4\left( \alpha -\beta \right) }\underset{n=0}{\overset{%
\infty }{\sum }}\left[ \left\{ 2\alpha \left[ bJ_{0}-J_{1}\right]
+4bJ_{0}\right\} \frac{4^{n+1}}{\left( \alpha +2\right) ^{n+1}}+\left\{
2\beta \left[ J_{1}-bJ_{0}\right] -4bJ_{0}\right\} \frac{4^{n+1}}{\left(
\beta +2\right) ^{n+1}}\right] x^{2n+1}
\end{equation*}%
which can be simplified as%
\begin{equation*}
\frac{-4^{n+1}}{4\left( \alpha -\beta \right) }\underset{n=0}{\overset{%
\infty }{\sum }}\left[ \frac{\left( \beta +2\right) ^{n+1}\left\{ 2\alpha %
\left[ bJ_{0}-J_{1}\right] +4bJ_{0}\right\} +\left( \alpha +2\right)
^{n+1}\left\{ 2\beta \left[ J_{1}-bJ_{0}\right] -4bJ_{0}\right\} }{\left(
\alpha +2\right) ^{n+1}\left( \beta +2\right) ^{n+1}}\right] x^{2n+1}
\end{equation*}%
$\beta +2=-\frac{\beta }{\alpha },$implies gives%
\begin{equation*}
\frac{-4}{4\left( \alpha -\beta \right) }\underset{n=0}{\overset{\infty }{%
\sum }}\left\{ 
\begin{array}{c}
\left[ bJ_{0}-J_{1}\right] \left[ \alpha \left( \alpha +2\right) ^{n}-\beta
\left( \beta +2\right) ^{n}\right] + \\ 
bJ_{0}\left[ \left( \beta +2\right) ^{n+1}-\left( \alpha +2\right) ^{n+1}%
\right]%
\end{array}%
\right\} x^{2n+1}
\end{equation*}%
with $\left( \alpha +2\right) =\frac{\alpha ^{2}}{ab}$, we simplify the
above expression as%
\begin{equation*}
\frac{-1}{\left( \alpha -\beta \right) }\underset{n=0}{\overset{\infty }{%
\sum }}\left( \frac{1}{ab}\right) ^{n+1}\left\{ 
\begin{array}{c}
ab\left[ bJ_{0}-J_{1}\right] (\alpha ^{2n+1}-\beta ^{2n+1}) \\ 
-bJ_{0}(\alpha ^{2n+2}-\beta ^{2n+2})%
\end{array}%
\right\} x^{2n+1}
\end{equation*}%
This can be further expanded and simplified as%
\begin{equation*}
\underset{n=0}{\overset{\infty }{\sum }}\left\{ \left( J_{1}-bJ_{0}\right) %
\left[ \frac{\alpha ^{2n+1}-\beta ^{2n+1}}{\left( ab\right) ^{n}\left(
\alpha -\beta \right) }\right] +bJ_{0}\left[ \frac{\alpha ^{2n+2}-\beta
^{2n+2}}{\left( ab\right) ^{n+1}\left( \alpha -\beta \right) }\right]
\right\} x^{2n+1}
\end{equation*}%
Now the even and the odd expressions obtained can be condensed by means of
the parity function $[...]$ as%
\begin{equation*}
J(x)=\underset{n=0}{\overset{\infty }{\sum }}\left\{ 
\begin{array}{c}
\left( J_{1}-bJ_{0}\right) ^{\epsilon (n)}\left(
aJ_{1}-2J_{0}-abJ_{0}\right) ^{1-\epsilon (n)}\left\{ \frac{\alpha
^{n}-\beta ^{n}}{\left( ab\right) ^{\left\lfloor \frac{n}{2}\right\rfloor
}\left( \alpha -\beta \right) }\right\} \\ 
+b^{\epsilon (n)}J_{0}\left\{ \frac{\alpha ^{2\left( \left\lfloor \frac{n}{2}%
\right\rfloor +1\right) }-\beta ^{2\left( \left\lfloor \frac{n}{2}%
\right\rfloor +1\right) }}{\left( ab\right) ^{\left\lfloor \frac{n}{2}%
\right\rfloor +1}\left( \alpha -\beta \right) }\right\}%
\end{array}%
\right\} x^{n}
\end{equation*}%
Therefore compared with $J(x)=\underset{n=1}{\overset{\infty }{\sum }}%
J_{n}x^{n}$ we obtain our Binet formula as%
\begin{equation*}
J_{n}=A\left( \alpha ^{n}-\beta ^{n}\right) +B\left( \alpha ^{2\left\lfloor 
\frac{n}{2}\right\rfloor +2}-\beta ^{2\left\lfloor \frac{n}{2}\right\rfloor
+2}\right)
\end{equation*}%
where%
\begin{equation*}
A=\frac{\left( J_{1}-bJ_{0}\right) ^{\epsilon (n)}\left(
aJ_{1}-2J_{0}-abJ_{0}\right) ^{1-\epsilon (n)}}{\left( ab\right)
^{\left\lfloor \frac{n}{2}\right\rfloor }\left( \alpha -\beta \right) }\text{
\ \ \ and \ }B=\text{\ }\frac{b^{\epsilon (n)}J_{0}}{\left( ab\right)
^{\left\lfloor \frac{n}{2}\right\rfloor +1}\left( \alpha -\beta \right) }.
\end{equation*}

\begin{theorem}
$\left( SUMMATION\text{ }FORMULA\right) $

For any positive integer n, we have%
\begin{equation*}
\underset{k=0}{\overset{n-1}{\sum }}J_{k}=\frac{J_{n}(1-a^{\xi (n)}b^{1-\xi
(n)})+2J_{n-1}(1-a^{1-\xi (n)}b^{\xi (n)})+J_{1}(a-1)+J_{0}(2b-ab-1)}{1-ab}.
\end{equation*}
\end{theorem}

\begin{proof}
If $n$ is even, we have%
\begin{eqnarray*}
\underset{k=0}{\overset{n-1}{\sum }}J_{k} &=&\underset{k=0}{\overset{\frac{%
n-2}{2}}{\sum }}J_{2k}+\underset{k=0}{\overset{\frac{n-2}{2}}{\sum }}J_{2k+1}
\\
&=&\underset{k=0}{\overset{\frac{n-2}{2}}{\sum }}\frac{aJ_{1}-2J_{0}-abJ_{0}%
}{\left( ab\right) ^{k}}\frac{\alpha ^{2k}-\beta ^{2k}}{\alpha -\beta }+%
\underset{k=0}{\overset{\frac{n-2}{2}}{\sum }}\frac{J_{0}}{\left( ab\right)
^{k+1}}\frac{\alpha ^{2k+2}-\beta ^{2k+2}}{\alpha -\beta } \\
&&+\underset{k=0}{\overset{\frac{n-2}{2}}{\sum }}\frac{J_{1}-bJ_{0}}{\left(
ab\right) ^{k}}\frac{\alpha ^{2k+1}-\beta ^{2k+1}}{\alpha -\beta }+\underset{%
k=0}{\overset{\frac{n-2}{2}}{\sum }}\frac{bJ_{0}}{\left( ab\right) ^{k+1}}%
\frac{\alpha ^{2k+2}-\beta ^{2k+2}}{\alpha -\beta }
\end{eqnarray*}%
\begin{eqnarray*}
\underset{k=0}{\overset{n-1}{\sum }}J_{k} &=&\frac{aJ_{1}-2J_{0}-abJ_{0}}{%
\left( ab\right) ^{\frac{n}{2}-1}(\alpha -\beta )}\left[ \frac{\alpha
^{n}-\left( ab\right) ^{\frac{n}{2}}}{(\alpha ^{2}-ab)}-\frac{\beta
^{n}-\left( ab\right) ^{\frac{n}{2}}}{(\beta ^{2}-ab)}\right] \\
&&+\frac{J_{0}}{\left( ab\right) ^{\frac{n}{2}}(\alpha -\beta )}\left[ \frac{%
\alpha ^{n+2}-\alpha ^{2}\left( ab\right) ^{\frac{n}{2}}}{(\alpha ^{2}-ab)}-%
\frac{\beta ^{n+2}-\beta ^{2}\left( ab\right) ^{\frac{n}{2}}}{(\beta ^{2}-ab)%
}\right] \\
&&+\frac{J_{1}-bJ_{0}}{\left( ab\right) ^{\frac{n}{2}-1}(\alpha -\beta )}%
\left[ \frac{\alpha ^{n+1}-\alpha \left( ab\right) ^{\frac{n}{2}}}{(\alpha
^{2}-ab)}-\frac{\beta ^{n+1}-\beta \left( ab\right) ^{\frac{n}{2}}}{(\beta
^{2}-ab)}\right] \\
&&+\frac{bJ_{0}}{\left( ab\right) ^{\frac{n}{2}}(\alpha -\beta )}\left[ 
\frac{\alpha ^{n+2}-\alpha ^{2}\left( ab\right) ^{\frac{n}{2}}}{(\alpha
^{2}-ab)}-\frac{\beta ^{n+2}-\beta ^{2}\left( ab\right) ^{\frac{n}{2}}}{%
(\beta ^{2}-ab)}\right] .
\end{eqnarray*}%
\begin{eqnarray*}
&=&\frac{(aJ_{1}-2J_{0}-abJ_{0})}{\left( ab\right) ^{\frac{n}{2}+1}(\alpha
-\beta )(1-ab)}\left[ 4a^{2}b^{2}(\alpha ^{n-2}-\beta ^{n-2})-ab(\alpha
^{n}-\beta ^{n})+\left( ab\right) ^{\frac{n}{2}}(\alpha ^{2}-\beta ^{2})%
\right] \\
&&+\frac{J_{0}}{\left( ab\right) ^{\frac{n}{2}+2}(\alpha -\beta )(1-ab)}%
\left[ 4a^{2}b^{2}(\alpha ^{n}-\beta ^{n})-ab(\alpha ^{n+2}-\beta
^{n+2})+\left( ab\right) ^{\frac{n}{2}+1}(\alpha ^{2}-\beta ^{2})\right] \\
&&+\frac{J_{1}-bJ_{0}}{\left( ab\right) ^{\frac{n}{2}+1}(\alpha -\beta
)(1-ab)}\left[ 4a^{2}b^{2}(\alpha ^{n-1}-\beta ^{n-1})-ab(\alpha
^{n+1}-\beta ^{n+1})-\left( ab\right) ^{\frac{n}{2}+1}(\alpha -\beta )\right]
\\
&&+\frac{bJ_{0}}{\left( ab\right) ^{\frac{n}{2}+2}(\alpha -\beta )(1-ab)}%
\left[ 4a^{2}b^{2}(\alpha ^{n}-\beta ^{n})-ab(\alpha ^{n+2}-\beta
^{n+2})+\left( ab\right) ^{\frac{n}{2}+1}(\alpha ^{2}-\beta ^{2})\right]
\end{eqnarray*}%
\begin{eqnarray*}
&=&\frac{-J_{n+1}-J_{n}+4J_{n-1}+4J_{n-2}+J_{1}(a-1)+J_{0}(2b-ab-1)}{1-ab} \\
&=&\frac{-J_{n+1}-J_{n}+4J_{n-1}+4J_{n-2}+J_{1}(a-1)+J_{0}(2b-ab-1)}{1-ab} \\
&=&\frac{J_{n}(1-b)+J_{n-1}(2-2a)+J_{1}(a-1)+J_{0}(2b-ab-1)}{1-ab}
\end{eqnarray*}%
similarly if $n$ is odd, we obtain 
\begin{eqnarray*}
\underset{k=0}{\overset{n-1}{\sum }}J_{k} &=&\underset{k=0}{\overset{\frac{%
n-1}{2}}{\sum }}J_{2k}+\underset{k=0}{\overset{\frac{n-3}{2}}{\sum }}J_{2k+1}
\\
&=&\frac{-J_{n+1}-J_{n}+4J_{n-1}+4J_{n-2}+J_{1}(a-1)+J_{0}(2b-ab-1)}{1-ab} \\
&=&\frac{J_{n}(1-a)+J_{n-1}(2-2b)+J_{1}(a-1)+J_{0}(2b-ab-1)}{1-ab}
\end{eqnarray*}%
Hence putting the two results together by means of the parity function gives 
\begin{equation*}
\underset{k=0}{\overset{n-1}{\sum }}J_{k}=\frac{J_{n}(1-a^{\xi (n)}b^{1-\xi
(n)})+2J_{n-1}(1-a^{1-\xi (n)}b^{\xi (n)})+J_{1}(a-1)+J_{0}(2b-ab-1)}{1-ab}.
\end{equation*}
\end{proof}

\begin{theorem}
$\left( SUMMATION\text{ }FORMULA\right) $

For any positive integer n, we have%
\begin{eqnarray*}
\underset{k=0}{\overset{n-1}{\sum }}\frac{J_{k}}{x^{k}} &=&\frac{%
J_{n}(2-x-a^{\xi (n)}b^{1-\xi (n)}x)+2J_{n-1}(2-a^{1-\xi (n)}b^{\xi (n)}-x)}{%
x^{2}-(ab+4)x+4} \\
&&+\frac{x^{2}(J_{1}-bJ_{0})+x(-2J_{1}+3bJ_{0}+aJ_{1}-J_{0}-abJ_{0})}{%
x^{2}-(ab+4)x+4}.
\end{eqnarray*}

\begin{proof}
The proof is obtained is a similar fashion as the above theorem.
\end{proof}
\end{theorem}

\end{document}